\documentclass[10pt,leqno]{article}

\usepackage{amsmath,amssymb,amsthm,mathrsfs,dsfont}

\usepackage[margin=3cm]{geometry} 

\usepackage{titlesec,hyperref}

\usepackage{color}

\usepackage{fancyhdr}
\usepackage{subfigure}
\usepackage[graphicx]{realboxes}
\usepackage{float}
\usepackage{enumerate}
\usepackage{enumitem}
\setlist[enumerate]{label=(\arabic*), leftmargin=*, itemsep=2pt}
\usepackage{dsfont} 
\usepackage{extarrows}

\titleformat{\subsection}{\it}{\thesubsection.\enspace}{1pt}{}

\newtheorem{theo}{Theorem}[section]
\newtheorem{lemm}[theo]{Lemma}
\newtheorem{defi}[theo]{Definition}
\newtheorem{coro}[theo]{Corollary}

\newtheorem{rema}[theo]{Remark}

\allowdisplaybreaks 

\numberwithin{equation}{section}

\begin{document}
\title{Invariant Measure of the Camassa–Holm Equation with Linear Multiplicative Noise
	\hspace{-4mm}
}

\author{
	Wei $\mbox{Luo}^1$ \footnote{E-mail:  luow23@mail.sysu.edu.cn}, \quad
	Zhaoyang $\mbox{Yin}^{1,2}$\footnote{E-mail: mcsyzy@mail.sysu.edu.cn} \quad and
	Pei $\mbox{Zheng}^{1}$\footnote{E-mail: zhengp25@mail2.sysu.edu.cn}\\
	$^1\mbox{Department}$ of Mathematics,
	Sun Yat-sen University, Guangzhou 510275, China\\
	$^2\mbox{School}$ of Science,\\ Shenzhen Campus of Sun Yat-sen University, Shenzhen 518107, China}

\date{}
\maketitle
\hrule

\begin{abstract}
	In this paper, we prove that the solution map of Camassa-Holm equation with linear multiplicative noise
	$$
		\left\{
		\begin{array}{l}
			{\rm d}u+(u\partial_xu+\partial_xP[u])\,{\rm d}t=\beta u\,{\rm d}W,\\
			u(0,x)=u_0(x),\\
			P[u]=(1-\partial_x^2)^{-1}\left(u^2+\frac 1 2(\partial_x u)^2\right)
		\end{array}
		\right.
	$$
	depends almost surely continuously on the deterministic initial data in $H^s$ for $s>3/2$. Furthermore, we prove the existence and non-uniqueness of an invariant measure for the Camassa-Holm equation with linear multiplicative noise.
	
	\vspace*{5pt}
					
	\vspace*{5pt}
	
\end{abstract}

\vspace*{10pt}

\tableofcontents

\section{Introduction}\label{1}
\quad In recent year, the Camassa-Holm (CH) equation obtained by Camassa and Holm the nonlinear partial differential equation \cite{CH}
$$
	m_{t}+um_{x}+2u_{x}m=0, m=u-u_{xx}+\kappa
$$
has been well studied and a series of achievements have been made. A particular feature of the CH equation is that when $\kappa=0$ 
$$
	u_t-u_{xxt}+3uu_x=2u_x\,u_{xx}+u\,u_{xxx}
$$
it admits peaked soliton solutions which are also called peakons. It can be regarded as a shallow water wave equation with nonhydrostatic pressure \cite{CH,CH2,CH3}
$$
	\left\{
	\begin{aligned}
		&u_t+uu_x+P_x=0,\\
		&P-P_{xx}=u^2+\frac 1 2u_x^2
	\end{aligned}
	\right.
$$

The bi-Hamiltonian structure of CH equation was studied in \cite{CH6,CH7}. Based on the bi-Hamiltonian structure, it ensures infinite conversation laws \cite {CH}, and the complete integrability was discussed in \cite{CH,CH4,CH5}. Furthermore, the CH equation exhibits both phenomena of (peaked) soliton interaction and wave breaking (the solution remains bounded but its slope becomes unbounded in finite time; cf. \cite{CH12}.)

Constantin and Escher \cite{CH13,CH14} investigated the Cauchy problem for the periodic Camassa–Holm equation. The wave breaking for Cauchy problem was studied in \cite{CH14,CH5,CH15}, and in \cite{CH16} Constantin claimed that wave breaking is the only way that singularities can occur in solutions. More precisely,the CH equation admits orbitally stable peaked solitons given explicitly by
$$u(t,x)=ce^{|x-ct|},\quad x\in\mathbb{R},\ c>0$$
as shown in \cite{CH17}. 

In recent year, the local well-posedness of Cauchy problem of the deterministic CH equation in Besov spaces and Sobolev spaces was proved in \cite{CH8,CH9,CH10,CH11}. Based on the local well-posedness in Besov and Sobolev spaces, Constantin and Molinet \cite{Constantin1999} proved the existence and uniqueness of global weak solutions satisfying energy conservation, both on the whole line and in the periodic case. They also established the stability of solitons. Using a viscous approximation, Xin and Zhang also proved the existence of CH equation in \cite{Xin2000}, and claimed the one-sided supernorm estimate and space-time integrability estimate in $L^p_{\rm loc}(\mathbb{R}^+\times\mathbb{R})$ for $p<3$ of the weak solutions, additionally, the large-time behavior of the weak solution is given. In the subsequent work \cite{Xin2002}, they established a "weak=strong" theorem for the admissible weak solutions and the uniqueness of weak solution under the condition that the initial data $m_0=(1-\partial_x^2)u_0$ is a positive Radon measure.

However, due to the uncertainties in geophysical and climate dynamics \cite{Arnold2001,Holm2015}, we have to consider some influence of internal, external, or environmental noises. Besides, the whole background for the considered physical system may be difficult to describe deterministically. Thus, we consider the randomness of the background movement which is one of the prevailing hypotheses on the onset of turbulence in fluid models, and there is a lot of recent work done on PDEs with random perturbations \cite{SPDE1,SPDE2,Debussche2011,SPDE3,SPDE4,SPDE5,SPDE6,SPDE7}.

The stochastic CH equation was derived via the stochastic variational method in \cite{Holm2015,Holm2016}. Consequently, the well-posedness, uniqueness and blow up phenomena for the stochastic Camassa-Holm type equations with perturbation is currently a interesting topic in the field of physics and mathematics. Chen and Gao established the existence of stochastic CH equation with additive noise in $H^m$ with $m>3/2$ in \cite{Chen2012}, and the existence of a pathwise solution to a modified CH equation with deterministic initial data and linear multiplicative noise in \cite{Chen2016}. In \cite{Tang2018}, Tang proved the local existence and pathwise uniqueness of pathwise solution in Sobolev space $H^m$ with $m>3/2$, and studied the condition lead to the global existence and the blow-up for the linear noise case. In particular, Tang studied the pathwise dissipative effect of the linear noise on the periodic peakons to the deterministic CH eqaution. The well-posedness for a generalized CH equation with higher order nonlinearities under random perturbation was studied by Miao, Rohde and Tang in \cite{Miao2024}. 

%

In this paper, we consider the initial-value problem for the stochastic CH equation with linear multiplicative noise and deterministic initial data:
\begin{equation}\label{CH u}
	\left\{
	\begin{array}{l}
		{\rm d}u+(u\partial_xu+\partial_xP[u])\,{\rm d}t=\beta u\,{\rm d}W,\\
		u(0,x)=u_0(x),\\
		P[u]=(1-\partial_x^2)^{-1}\left(u^2+\frac 1 2(\partial_x u)^2\right)
	\end{array}
	\right.
\end{equation}
for $(t,x)\in [0,T]\times\mathbb{S}$, where $\beta\in\mathbb{R}$ with $\beta\neq0$, $\mathbb{S}=\mathbb{S}/(2\pi\mathbb{Z})$ is the 1D torus, $T$ is a positive final time and $W$ is a 1D Wiener process defined on a standard filtered probability space $\mathcal{S}=(\Omega,\mathcal{F},\{\mathcal{F}_t\}_{t\in [0,T]},\mathbb{P})$, henceforth called a stochastic basis. In particular, we assume that the initial value $u_0$ is a deterministic function independent of any random quantities in the rest of this paper. Moreover, the elliptic equation for $P$ can be solved to supply
\begin{equation}\label{P convolution}
	P=P[u]:=K\ast\left(u^2+\frac 1 2(\partial_x u)^2\right),\quad K(x)=\frac{\cosh(x-2\pi\,{\rm int}(\frac{x}{2\pi})-\pi)}{2\sinh\,\pi},
\end{equation}
where $K$ is the Green's function of $1-\partial_x^2$ on $\mathbb{S}$, ${\rm int}(x)$ is the integer part of $x$, and $\ast$ means convolution in $x$.

During the last few decades, invariant measures and ergodicity for global solutions of equations with perturbations have been widely studied. 
The existence of an invariant measure for the Burgers equation perturbed by a space-time white noise was established in \cite{DaPrato1994}. 
Subsequently, Flandoli and coauthors proved the ergodicity for the stochastic incompressible Navier-Stokes equation with additive noise in bounded two-dimensional In recent years, invariant measures and ergodicity for global solutions of equations with perturbations have been widely studied. domains \cite{Flandoli1994,Flandoli1995}. 
In three dimensions, Da Prato and Debussche \cite{DaPrato2003} studied the ergodicity of the stochastic Navier-Stokes equation with additive noise. 
Cerrai \cite{Cerrai2003} obtained the existence of an invariant measure for a class of reaction-diffusion systems perturbed by multiplicative noise in bounded domains of dimension $d \ge 1$. 
Hairer and Mattingly \cite{Hairer2006} established the ergodicity of the stochastically forced Navier-Stokes equations on the two-dimensional torus with Brownian forcing. 
Brze\'zniak and Li \cite{Brzezniak2006} proved the existence of an invariant measure for the two-dimensional Navier-Stokes equations in a domain satisfying the Poincar\'e inequality, perturbed by an additive irregular noise. 
Dong and Xu \cite{Dong2007} investigated the existence of an invariant measure for the Burgers equation driven by a Poisson process and a combined Poisson-Wiener noise on the one-dimensional circle. 
More recently, \cite{Brzezniak2016} showed the existence of invariant measures for the stochastic extensible beam equation and the stochastic damped wave equation with polynomial nonlinearities. 
Brze\'zniak, Motyl, and Ondrej\'at \cite{Brzezniak2017} proved the existence of an invariant measure for the stochastic two-dimensional Navier-Stokes equations with multiplicative noise in unbounded domains.

Despite these important developments, the continuous dependence on initial data and the questions of existence and uniqueness of invariant measures for the Camassa–Holm equation with linear multiplicative noise remain largely unexplored. In this paper, we are mainly concerned with the continuous dependence on initial data for the CH equation with linear multiplicative noise, and then, based on this result, we study the existence and non-uniqueness of invariant measures for the stochastic CH equation. Our main result is the following theorem:
\begin{theo}\label{main}
	Let $s>3$ and $\mathcal{S}=(\Omega,\mathcal{F},\mathbb{P},\{\mathcal{F}_t\}_{t\ge0},W)$ be a fixed stochastic basis. Denote the Polish space of deterministic functions
	\begin{equation}\label{Es}
		E_s=\{u_0\in H^s|(1-\partial_x^2)u_0\ge0\ {\rm or}\ (1-\partial_x^2)u_0\le0,\quad \forall x\in\mathbb{S}\}
	\end{equation}
	endowed with the metric induced by the $H^s$-norm. Then equation \eqref{CH u} admits an invariant probability measure. Moreover, this invariant measure is not unique.
\end{theo}

\section{Notations and Preliminaries}\label{2}

\quad In this paper, $\mathcal{S}=(\Omega,\mathcal{F},\mathbb{P},\{\mathcal{F}_t\}_{t\ge 0},W)$ is a stochastic basis, where $\mathbb{P}$ is a probability measure on $\Omega$, $\mathcal{F}$ is a $\sigma$-algebra, $\{\mathcal{F}_t\}_{t\ge 0}$ is a right-continuous filtration on $(\Omega,\mathcal{F})$ such that $\{\mathcal{F}_0\}$ contains all the $\mathbb{P}$-negligible subsets, and $W(t)=W(\omega,t)$, $\omega\in\Omega$ is a standard 1-D Brownian motion defined on $(\Omega,\mathcal{F},\mathbb{P},\{\mathcal{F}_t\}_{t\ge 0})$.

\begin{defi}[pathwise solution]
	Let $\mathcal{S}=(\Omega,\mathcal{F},\mathbb{P},\{\mathcal{F}_t\}_{t\ge 0},W)$ be a fixed stochastic basis. Let $s>3/2$ and $u_0$ be an $H^s$-valued $\mathcal{F}_0$ measurable random variable relative to $\mathcal{S}$. A local pathwise $H^s$ solution to \eqref{CH u} is a pair $(u,\tau)$, where $\tau$ is a stopping time satisfying $\mathbb{P}(\tau>0)=1$ and $u:\Omega\times[0,\tau]\rightarrow H^s$ is an $\mathcal{F}_t$ predictable $H^s$-valued process satisfying:
	\begin{enumerate}
		\item $\mu_0(Y)=\mathbb{P}(u(0)\in Y)$ for all $Y\in\mathcal{B}(H^s)$ and $u(\cdot\wedge\tau)\in L^2(\Omega;C([0,\infty);H^s))$ and $u(\cdot\wedge\tau)\in C([0,\infty);H^s)$, $\mathbb{P}$-a.s.
		\item For every $t>0$ and $v\in C^\infty(\mathbb{S})$,
		$$(u(t\wedge\tau),v)_{L^2}-(u(0),v)_{L^2}+\int_0^{t\wedge\tau}(u\partial_xu+\partial_xP[u],v)_{L^2}\,{\rm d}s=\int_0^{t\wedge\tau}\beta (u,v)_{L^2}\,{\rm d}W$$
		almost surely.
	\end{enumerate}
	
	If given any two pairs of local pathwise solutions $(u_!,\tau_1)$ and $(u_2,\tau_2)$ with $\mathbb{P}(u_1(0)=u_2(0))=1$ satisfy
	$$\mathbb{P}(u_1(t,x)=u_2(t,x),\forall (t,x)\in[0,\tau_1\wedge\tau_2]\times\mathbb{S})=1,$$
	then the local pathwise solutions are said to be pathwise unique.
\end{defi}

\begin{defi}[maxial and global solutions]
	Let $s>3/2$. A maximal $H^s$ solution to \eqref{CH u} is a triple $(u,\{\tau_n\}_{n\in\mathbb{N}},\xi)$ such that for each $n\in\mathbb{N}$, $(u,\tau_n)$ is a pathwise $H^s$ solution, $\tau_{n+1}\ge \tau_n$, $\lim_{n\rightarrow\infty}\tau_n=\xi$, and
	$$\sup_{t\in[0,\tau_n]}\|u\|_{H^s}\ge n,\ \mathbb{P}-\text{a.s., on the set }\{\xi<\infty\}.$$
	If $\xi=\infty$, $\mathbb{P}$-a.s., then we say that the pathwise solution exists globally.
\end{defi}

Now we briefly recall some relevant mathematical preliminaries from functional analysis and probability theory. 

\begin{lemm}\cite{BCD2011,Tang2018}\label{P estimate}
	For any $u,v\in H^s$ with $s>1/2$, we have
	$$\|P_x[v]\|_{H^s}\lesssim(\|v\|_{L^\infty}+\|\partial_x v\|_{L^\infty})\|v\|_{H^s},\quad s>3/2,$$
	$$\|P_x[u]-P_x[v]\|_{H^s}\lesssim(\|u\|_{H^s}+\|v\|_{H^s})\|u-v\|_{H^s},\quad s>3/2,$$
	$$\|P_x[u]-P_x[v]\|_{H^s}\lesssim(\|u\|_{H^{s+1}}+\|v\|_{H^{s+1}})\|u-v\|_{H^s},\quad 1/2<s<3/2.$$
\end{lemm}


To construct the invariant measure of \eqref{CH u}, we first introduce the following notations and a key theorem.

\begin{defi}
	Let $E$ be a Polish space, denote $C_b(E)$ is the set of all real continuous and bounded Borel on $E$, $\mathcal{M}_1(E)$ is the set of all probability measures defined on $(E,\mathcal{B}(E))$. If for all $t\ge 0$, $x\in E$ and $\Gamma\in \mathcal{B}(E)$, we have:
	\begin{enumerate}
		\item $P_t(x,\cdot)$ is a probability measure on $(E,\mathcal{B}(E))$.
		\item  $P_t(\cdot,\Gamma)$ is an $\mathcal{B}(E)$-measurable function.
		\item $P_{t+s}(x,\Gamma)=\int_E P_s(y,\Gamma)P_t(x,{\rm d}y)$.
		\item $P_0(x,\Gamma)=\chi_\Gamma(x)$.
	\end{enumerate}
	Then $P_t$ is a Markovian transition.
	
	Any transition function $P_t(x,\Gamma)$ defines a semigroup of linear operators $P_t$, $t\ge 0$ on the space $C_b(E)$ by the formula
	$$P_t\varphi(x)=\int_E \varphi(y)P_t(x,{\rm d}y),\quad t\ge 0,\ x\in E,\ \varphi\in C_b(E),$$
	$P_t$ is called the Markovian transition semigroup associated to the transition function $P_t(x,\Gamma)$.
	
	In particular, a Markovian semigroup $P_t$ is said to be stochastically continuous if 
	$$\lim_{t\rightarrow0} P_t(x,B(x,\delta))=1,\quad\text{for all }x\in E,\ \delta>0.$$
\end{defi}

\begin{defi}[Markovian Feller semigroup]
	A stochastically continuous Markovian semigroup $P_t$, $t\ge 0$, is called a Markovian Feller semigroup, if for any $\varphi\in C_b(E)$ and $t\ge 0$ one has $P_t\varphi\in C_b(E)$.
\end{defi}

\begin{theo}[Krylov-Bogoliubov Theorem]\label{KB theorem}\cite{DaPrato1996,Krylov1973}
	If $P_t$, $t\ge0$ is a stochastically continuous Markovian semigroup, then for every $x\in E$ and $T>0$ the formula
	$$\mu_T(x,\Gamma)=\frac 1 T\int_0^T P_t(x,\Gamma)\,{\rm d}t,\quad\Gamma\in\mathcal{B}(E),$$
	defines a probability measure. For any $\nu\in\mathcal{M}_1(E)$, $R^\ast_T\nu$ is defined by
	$$R^\ast_T\nu(\Gamma)=\int_ER_T(x,\Gamma)\,\nu({\rm d}x),\quad\Gamma\in\mathcal{B}(E).$$
	
	If for some $\nu\in\mathcal{M}_1(E)$ and some sequence $T_n\uparrow\infty$ the sequence $\{R^\ast_{T_n}\nu\}$ is tight, then there exists an invariant measure for $P_t$.
\end{theo}

According to the main results from \cite{Tang2018}, the existence and uniqueness of global pathwise solutions to \eqref{CH u} was obtained:
\begin{theo}\label{global}
	Let $s>3$ and $\mathcal{S}=(\Omega,\mathcal{F},\mathbb{P},\{\mathcal{F}_t\}_{t\ge0},W)$ be a fixed stochastic basis. Assume that $u_0\in H^s$ is a deterministic function and $(1-\partial_x)^2u_0(x)>0$ or $(1-\partial_x)^2u_0(x)<0$ for all $x\in\mathbb{S}$, then for a.e. $\omega\in\Omega$, there exists a unique global pathwise solution $u$ to \eqref{CH u}, i.e. $\mathbb{P}(u\text{ exists globally})=1$. 
	Moreover, denote $\tilde{u}=\eta^{-1}(\omega,t)u$ with $\eta(\omega,t)=e^{\beta W(t)-\frac{\beta^2} 2 t}$,
	\begin{equation}\label{u Girsanov H^1}
		\mathbb{P}(\|\tilde{u}(t)\|_{H^1}=\|\tilde{u}_0\|_{H^1}=\|u_0\|_{H^1},\ \forall t>0)=1,
	\end{equation}
	\begin{equation}\label{u H^1}
		\mathbb{P}(\|u(t)\|_{H^1}\le \eta(\omega,t)\|u_0\|_{H^1},\ \forall t>0)=1,
	\end{equation}
	\begin{equation}\label{u W^1,infty}
		\mathbb{P}(\|u(t)\|_{W^{1,\infty}}\le 2\eta(\omega,t)\,\|u_0\|_{H^1},\ \forall t>0)=1.
	\end{equation}
	In particular, for all $T\ge 0$,
	\begin{equation}\label{u H^s}
		\mathbb{P}(\|u(t)\|_{H^s}\le C_s\eta(\omega,t)\|u_0\|_{H^s}\exp(C_s\int_0^T\eta(\omega,t)\,\|u_0\|_{H^1}\,{\rm d}t),\ \forall 0<t\le T)=1
	\end{equation}
	In particular, ${\rm sign}(u)={\rm sign}((1-\partial_x^2)u)={\rm sign}\left((1-\partial_x^2)u_0\right).$
\end{theo} 

\section{Continuity with respect to the initial data}
\begin{theo}
	Let $s>3/2$ and $\mathcal{S}=(\Omega,\mathcal{F},\mathbb{P},\{\mathcal{F}_t\}_{t\ge0},W)$ be a fixed stochastic basis. Suppose $\{u_0^k(x)\}_{k\in\mathbb{N}}$ be a sequence of $H^s$-valued deterministic functions such that $u_0^k\rightarrow u_0^\infty$ in $H^s$. For each $k\in\mathbb{N}\cup\{\infty\}$, let $(u^k,\{\tau_n^{k}\}_{n\in\mathbb{N}},\xi^k)$ be the unique maximal solution to \eqref{CH u} with initial data $u_0^k$. Then, for a.e. $\omega\in\Omega$, any $T>0$ there exists $\tilde{\xi}>0$ such that for all $0\le t<\tilde{\xi}\wedge T$,
	\begin{equation}\label{continuity H^s}
		\lim_{n\rightarrow\infty}\sup_{0\le s\le T}\|u^k(\omega,s)-u^\infty(\omega,s)\|_{H^s}=0.
	\end{equation}
	In other words, the solution map depends continuously on the initial data in $H^s$ almost surely.
\end{theo}

\begin{proof}
	For any $u_0,v_0\in H^s$, there exists two maximal pathwise solutions $(u,\{\tau_n^u\}_{n\in\mathbb{N}},\xi^u)$ and $(v,\{\tau_n^v\}_{n\in\mathbb{N}},\xi^v)$ to \eqref{CH u} with initial data $u_0$ and $v_0$, respectively. Define the stopping time
	$$\tau^w_n=\inf\{t>0:\max(\|u\|_{H^s},\|v\|_{H^s})> n\},$$
	we can easily deduce that for $n$ large enough $\mathbb{P}(\tau_n^w)>0$ and $\tau_n^w\le \tau^u\wedge\tau^v$. 
	
	Consider the Girsanov-type transform $\tilde{u}=\eta^{-1}(\omega,t)\,u$ and $\tilde{v}=\eta^{-1}(\omega,t)\,v$ with $\eta(\omega,t)=e^{\beta W(t)-\frac{\beta^2} 2 t}$. It is easy to obtain that
	$$
	\left\{
	\begin{aligned}
		&\partial_t\tilde{u}+\eta\,\tilde{u}\partial_x\tilde{u}+\eta\,\partial_x P[\tilde{u}]=0,\\
		&\tilde{u}(\omega,0,x)=u_0(x),
	\end{aligned}
	\right.
	\quad 
	\left\{
	\begin{aligned}
		&\partial_t\tilde{v}+\eta\,\tilde{v}\partial_x\tilde{v}+\eta\,\partial_x P[\tilde{v}]=0,\\
		&\tilde{v}(\omega,0,x)=v_0(x),
	\end{aligned}
	\right.
	$$
	and $\tilde{u},\tilde{v}\in C([0,\xi^u\wedge\xi^v);H^s)\cap C^1([0,\xi^u\wedge\xi^v);H^{s-1})$, $\mathbb{P}$-a.s.
	
	Denote $\tilde{w}=\tilde{u}-\tilde{v}=\eta^{-1}(u-v)$, $\tilde{h}=\tilde{u}+\tilde{v}=\eta^{-1}(u+v)$, we can see
	$$
	\left\{
	\begin{aligned}
		&\partial_t\tilde{w}+u\,\partial_x\tilde{w}+\tilde{w}\,\partial_x v+\eta\,\partial_x P(\tilde{w},\tilde{h})=0,\\
		&\tilde{w}(\omega,0,x)=u_0(x)-v_0(x),
	\end{aligned}
	\right.
	$$
	where $P(a,b)=K\ast (a\cdot b+\frac 1 2 a_x\cdot b_x)$. According to Lemma \ref{P estimate}, we have
	$$\|\tilde{w}\,\partial_xv\|_{H^{s-1}}\le C\|\tilde{w}\|_{H^{s-1}}\|v\|_{H^s}\quad{\rm and}\quad\|\eta\,\partial_x P(\tilde{w},\tilde{h})\|_{H^{s-1}}\le C\|\tilde{w}\|_{H^{s-1}}(\|u\|_{H^s}+\|v\|_{H^s}).$$
	Employing the energy estimates of the standard deterministic transport equation (cf. \cite{BCD2011}),w e obtain that for any $t\in[0,\tau_n^w\wedge T]$,
	$$\|\tilde{w}(\omega,t)\|_{H^{s-1}}\le\|u_0-v_0\|_{H^{s-1}}\exp\left(C\int_0^{\tau_n^w\wedge T}\|u(t)\|_{H^s}+\|v(t)\|_{H^s}\,{\rm d}t\right),$$
	thus,
	$$\|w(\omega,t)\|_{H^{s-1}}\le\eta(\omega,t)\,\|u_0-v_0\|_{H^{s-1}}\exp\left(C\int_0^{\tau_n^w\wedge T}\|u(t)\|_{H^s}+\|v(t)\|_{H^s}\,{\rm d}t\right).$$
	
	For fixed $\omega\in\Omega$, there exists $0<A(\omega)=\sup_{t>0}\eta(\omega,t)<\infty$, then for $t\in[0,\tau_n^w\wedge T]$,
	\begin{equation}\label{continuous s-1}
		\|u(\omega,t)-v(\omega,t)\|_{H^{s-1}}\le C\left(\tau_n^w\wedge T,n,A(\omega),\|u_0\|_{H^s},\|v_0\|_{H^s}\right)\|u_0-v_0\|_{H^{s-1}}.
	\end{equation}
	Therefore, we can conclude that for a.e. $\omega\in\Omega$, $t\in[0,\tau_n^w]$, $u(\omega,t)\rightarrow v(\omega,t)$ in $H^{s-1}$ as $u_0\rightarrow v_0$ in $H^{s-1}$.
	
	Denote $\bar{\mathbb{N}}=\mathbb{N}\cup\{\infty\}$, for each $k\in\bar{\mathbb{N}}$, consider the Girsanov-type transform $\tilde{u}^k$ of the maximal solutions $(u^k,\{\tau_n^k\}_{n\in\mathbb{N}},\xi^k)$ 
	\begin{equation}\label{tilde uk equation}
		\left\{
		\begin{aligned}
			&\partial_t\tilde{u}^k+\eta\,\tilde{u}^k\partial_x\tilde{u}^k+\eta\,\partial_x P[\tilde{u}^k]=0,\\
			&\tilde{u}^k(\omega,0,x)=u_0^k(x),
		\end{aligned}
		\right.
	\end{equation}
	and $\tilde{u}^k\in C([0,\xi^k);H^s)\cap C^1([0,\xi^k);H^{s-1})$, $\mathbb{P}$-a.s. Without loss of generality, define $\tau_n^k$ precisely as the stopping time:
	$$\tau_n^k:=\inf\left\{t>0:\|u^k(t)\|_{H^s}>n\right\}.$$
	And define the stopping time $\tau_n$ by
	$$\tau_n:=\inf\left\{t>0:\sup_{k\in\bar{\mathbb{N}}}\|u^k(t)\|_{H^s}>n\right\}=\inf\{\tau_n^k,\ k\in\bar{\mathbb{N}}\}.$$
	Since $\|u_0^k-u_0^\infty\|_{H^s}\rightarrow 0$, the sequence $\{u_0^k\}$ is uniformly bounded in $H^s$, i.e., $R_0:=\sup_{k\in\bar{\mathbb{N}}}\|u_0^k\|_{H^s}<\infty$. From this uniform boundedness, one can conclude that $\mathbb{P}(\tau_n>0)=1$ for sufficiently large $n$, and $0<\tau_n\le\inf_{k\in\bar{\mathbb{N}}}\tau_n^k\le\inf_{k\in\bar{\mathbb{N}}}\xi^k$.
	
	According to \eqref{continuous s-1}, $\eta\,\tilde{u}^k=u^k\rightarrow u^\infty=\eta\,\tilde{u}^\infty$ in $C([0,\tau_n];H^{s-1})$ for each $k\in\mathbb{N}$, and for all $t\in [0,\tau_n]$,
	$$\eta\,\|\partial_x P[\tilde{u}^k]-\eta\,\partial_x P[\tilde{u}^k]\|_{H^{1/2}}\le C_s\,\|\tilde{u}^k-\tilde{u}^\infty\|_{H^{1/2}}\,(\|u^k\|_{H^s}+\|u^\infty\|_{H^s})\le C(s,n)\,\|\tilde{u}^k-\tilde{u}^\infty\|_{H^{1/2}},$$
	$$\eta\,\|\partial_x P[\tilde{u}^k]-\eta\,\partial_x P[\tilde{u}^k]\|_{H^s}\le C_s\,\|\tilde{u}^k-\tilde{u}^\infty\|_{H^s}\,(\|u^k\|_{H^s}+\|u^\infty\|_{H^s})\le C(s,n)\,\|\tilde{u}^k-\tilde{u}^\infty\|_{H^s},$$
	by Kato's local wellposedness result for deterministic quasilinear evolution equations (see \cite{Danchin2001,Himonas2001,Kato1975,Li2000,RB2001}), we can deduce that
	$$\lim_{k\rightarrow\infty}\sup_{t\in[0,\tau_n]}\|\tilde{u}^k(\omega,t)-\tilde{u}^\infty(\omega,t)\|_{H^s}=0,$$
	since $\eta(\omega,t)\le A(\omega)$ for all $t\ge 0$, we obtain that 
	$$\lim_{n\rightarrow\infty}\sup_{t\in[0,\tau_n]}\|u^k(\omega,t)-u^\infty(\omega,t)\|_{H^s}=0.$$
	
	Finally, for sufficiently large $n$, $\tau_n>0$, $\tau_n\le\tau_{n+1}$, and $\tau_n\le\xi$. Applying the monotone convergence theorem for bounded monotone sequences, $\tilde{\xi}=\lim_{n\rightarrow}\tau_n>0$. For any $T>0$ and $0\le t<\tilde{\xi}\wedge T$, there exist $n_t\in\mathbb{N}$ such that $\tau_{n_t-1}<t\le\tau_{n_t}$. This concludes the proof of the theorem.
\end{proof}

\begin{coro}\label{continuity Es}
	Let $s>3$, $\mathcal{S}=(\Omega,\mathcal{F},\mathbb{P},\{\mathcal{F}_t\}_{t\ge0},W)$ be a fixed stochastic basis, and $E_s$ defined by \eqref{Es}. For each $u_0\in E_s$, there exists a unique global pathwise solution $u\in C([0,\infty);E_s)$ to\eqref{CH u} with initial data $u_0$, $\mathbb{P}$-a.s. Moreover, for any $T>0$, the solution map depends almost surely continuously on the initial data in $E_s$ in the sense of the $H^s$-norm, i.e., if $u_0^n\rightarrow u_0^\infty$ in $E_s$, then
	$$\sup_{t\in[0,T]}\|u^n(t)-u^\infty(t)\|_{H^s}\rightarrow0,\quad\mathbb{P}-{\rm a.s.}$$
\end{coro}
\begin{proof}
	The embedding $H^{s-2}\hookrightarrow L^\infty$ readily implies that $E_s$ is Polish space.
	
	Denote $M:=sup_{n\in\bar{\mathbb{N}}}\|u_0^n\|_{H^1}<\infty$, $A(\omega)=\sup_{t\ge 0}\eta(\omega,t)<\infty$, according to \eqref{u H^s}, we have for a.e. $\omega\in\Omega$, any $T>0$,
	$$\|u^n(\omega,t)\|_{H^s}\le C_sA(\omega)M\exp(C_sA(\omega)MT).$$
	Therefore, $\lim_{n\rightarrow\infty} \tau_n\wedge T=T$. The conclusion of the corollary follows directly from Theorem \ref{continuity H^s}.
\end{proof}

With the global pathwise well-posedness of the CH equation and the continuous dependence on initial data established, the transition probability family $\{P_t\}_{t\ge0}$ induces by these global solutions constitute a Markovian Feller semigroup on $E_s$.

\begin{coro}
	Let $s>3$, $\mathcal{S}=(\Omega,\mathcal{F},\mathbb{P},\{\mathcal{F}_t\}_{t\ge0},W)$ be a fixed stochastic basis, and $E_s$ defined by \eqref{Es}. The Markov transition probability measure generated by the CH equation \eqref{CH u}, defined by
	\begin{equation}\label{transition measure}
		P_t(u_0,A)=\mathbb{P}(u(t,x)\in A|u(0)=u_0),\quad\forall T\ge0,A\in\mathcal{B}(E_s),
	\end{equation}
	and the corresponding Markov semigroup defined as
	\begin{equation}\label{semigroup}
		P_t \varphi(u_0) := \mathbb{E}\bigl[ \varphi(u(t; u_0)) |u(0)=u_0\bigr]=\int_\Omega \varphi(y) P_t(u_0,{\rm d}y), \quad t \ge 0,\; u_0 \in E_s,\; \varphi \in C_b(E_s).
	\end{equation}
	
	Then the semigroup $P_t$ is a Markovian Feller semigroup on $E_s$.
\end{coro}
\begin{proof}
	According to Corollary \ref{continuity Es}, if $\{u_0^n\}_{n\in\mathbb{N}},u_0\in E_s$ and $u_0^n\rightarrow u_0\in H^s$, then for any $T>0$,
	$$\sup_{t\in[0,T]}\|u^n(t)-u(t)\|_{H^s}\rightarrow0,\quad\mathbb{P}-{\rm a.s.},$$
	where $u^n$ and $u$ are unique global pathwise solutions to \eqref{CH u} with initial data $u_0^n$ and $u_0$, respectively.
	For any $\varphi\in C_b(E_s)$, by the continuity and boundedness of $\varphi$, the Lebesgue dominated convergence theorem yields
	$$
	\mathbb{E}\left[ \varphi(u^n(t)) \right] \to \mathbb{E}\left[ \varphi(u(t)) \right],
	$$
	that is, $P_t \varphi(u_0^n) \to P_t \varphi(u_0)$. The uniqueness and continuity immediately imply that $P_t$ is a Markovian Feller semigroup, we omit the detail for brevity.
\end{proof}

\section{The existence of invariant measure}
\quad First, we introduce a lemma that will be very important in the subsequent proof.
\begin{lemm}\label{W-t/2}\cite{Fitz2020}
	$$\sup_{t\ge 0} (W(t)-\mu\,t)\overset{d}{=}e(\mu),$$
	where $W(t)$ is a Brownian motion, $\mu>0$ a drift and $e(\mu)$ is a random variable which has
	the exponential distribution with rate $2\mu$.
\end{lemm}

\begin{lemm}
	Let $s>3$ and $\mathcal{S}=(\Omega,\mathcal{F},\mathbb{P},\{\mathcal{F}_t\}_{t\ge0},W)$ be a fixed stochastic basis. Let $u_0\in E_s$, denote the time-averaged measure, 
	\begin{equation}\label{time average measure}
		\mu_T(A)=\frac 1 T\int_0^T P_t(u_0,A)\,{\rm d}t,\quad\forall T\ge0,A\in\mathcal{B}(E_s),
	\end{equation}
	where $P_t(u_0,A)$ is the transition probability measure defined as in \eqref{transition measure}. 
	
	Then for any $\delta>0$, $u_0\in H^{s+\delta}$, the family $\{\mu_T\}_{T>0}$ is tight.
\end{lemm}
\begin{proof}
	With out loss of generality, we may assume $\|u_0\|_{H^1}>0$. For fixed $\omega\in\Omega$, according to\eqref{u H^s}, for any $T>0$ and $t\in[0,T]$, we have
	$$\|u(t)\|_{H^{s+\delta}}\le C\eta(\omega,t)\,\|u_0\|_{H^{s+\delta}}\,\exp\left(\int_0^t\eta(\omega,s)\,\|u_0\|_{H^1}{\rm d}s\right).$$
	
	Denote $A(\omega)=\sup_{t>0}\eta(\omega,t)$,  $T(\omega)=\inf\{T:\beta W(t)-\frac{\beta^2} 2\, t\le-\frac{\beta^2} 4\, t,\ \forall t\ge T\}$, and $B_R=\{u\in E_s\bigg|\|u\|_{H^{s+\delta}}\le R\}$, then we have
	$$
		\begin{aligned}
			\mu_T(B_R^c)&=\frac 1 T\int_0^T\mathbb{P}\left(\|u\|_{H^{s+\delta}}>R|u(0)=u_0\right)\,{\rm d}t\\
			&\le \frac 1 T\int_0^T\mathbb{P}\left(C\eta(\omega,t)\|u_0\|_{H^{s+\delta}}\exp\left(\int_0^t \eta(\omega,s)\|u_0\|_{H^1}\,{\rm d}s\right)>R\right)\,{\rm d}t\\
			&\le \frac 1 T \int_0^T\mathbb{P}\left( CA(\omega)\,\|u_0\|_{H^{s+\delta}}\,\exp\left(\|u_0\|_{H^1}(T(\omega)\cdot A(\omega)+1)\right)>R \right)\,{\rm d}t.
		\end{aligned}
	$$
	And we have
	$$
	\begin{aligned}
		&\mathbb{P}\left(A(\omega)\exp\left(\|u_0\|_{H^1}(T(\omega)\cdot A(\omega)+1)\right)>\frac R {C\|u_0\|_{H^{s+\delta}}}\right)\\
		\le\ 
		&\mathbb{P}\left(A(\omega)>\left(\frac R {C\|u_0\|_{H^{s+\delta}}}\right)^{\frac 1 2}\right)+\mathbb{P}\left(T(\omega)\cdot A(\omega)>\frac 1 {2\|u_0\|_{H^1}} \log\frac{R}{C\|u_0\|_{H^{s+\delta}}}-1\right)\\
		\le\ 
		&\mathbb{P}\left(A(\omega)>\left(\frac R {C\|u_0\|_{H^{s+\delta}}}\right)^{\frac 1 2}\right)+\mathbb{P}\left(T(\omega)>\left[\frac 1 {2\|u_0\|_{H^1}} \log\frac{R}{C\|u_0\|_{H^{s+\delta}}}-1\right]^{\frac 1 2}\right)\\
		+&\mathbb{P}\left(A(\omega)>\left[\frac 1 {2\|u_0\|_{H^1}} \log\frac{R}{C\|u_0\|_{H^{s+\delta}}}-1\right]^{\frac 1 2}\right),
	\end{aligned}
	$$
	by taking $R$ sufficiently large, the above constants are well-defined. Without loss of generality, we may assume $\beta>0$, then for the drifted Brownian motion $\tilde{W}=\sup_{t\ge 0}\left(W(t)-\frac\beta 2\,t\right)$, according to Lemma \ref{W-t/2}, we have 
	$$\mathbb{P}(\tilde{W}>m)=e^{-\beta m},$$
	thus
	\begin{equation}\label{P(A) inequality1}
		\mathbb{P}\left(A(\omega)>\left(\frac R {C\|u_0\|_{H^{s+\delta}}}\right)^{1/2}\right)\le CR^{-1/2},
	\end{equation}
	and
	\begin{equation}\label{P(A) inequality2}
		\mathbb{P}\left(A(\omega)>\left[\frac 1 {2\|u_0\|_{H^1}} \log\frac{R}{C\|u_0\|_{H^{s+\delta}}}-1\right]^{\frac 1 2}\right)\le\left[\frac 1 {2\|u_0\|_{H^1}} \log\frac{R}{C\|u_0\|_{H^{s+\delta}}}-1\right]^{-1/2}.
	\end{equation}
	
	Consider $\mathbb{P}(T(\omega)>M_2)$, we have $\mathbb{P}(T(\omega)>M_2)=\mathbb{P}\left(\sup_{t\ge M_2}(W(t)-\beta t/4)>0\right)$. Define $s=t-M_2\ge 0$, $W(M_2)=x$, then the strong Markov property of Brownian motion yields $W(M_2+s)=x+\overline{W}(s)$, where $\overline{W}(s)$ is a Brownian motion. Applying Lemma \ref{W-t/2}, we obtain that
	$$\sup_{t\ge M_2}\left(W(t)-\beta t/4\right)=\sup_{s\ge 0}\left(x+\overline{W}(s)-\frac \beta 4(M_2+s)\right)=x-\frac{\beta M_2}{4}+\sup_{s\ge0}(\overline{W}(s)-\beta s/4),$$
	$$\begin{aligned}
		\mathbb{P}\left(T(\omega)>M_2\bigg| W(M_2)=x\right)
		&=\mathbb{P}\left(\sup_{s\ge0}\left(\beta\overline{W}(s)-\beta^2 s/4\right)>\beta^2 M_2/4-\beta x\right)\\
		&=
	\left\{
	\begin{array}{ll}
		1,&\beta M_2/4-x<0,\\
		e^{-\beta^2 M_2/8+\beta x/2},&\beta M_2/4-x\ge 0,
	\end{array}
	\right.
	.
	\end{aligned}$$
	and $W(M_2)\overset{d}{=}N(0,M_2)$.
	
	Hence, using the standard method for computing conditional expectations \cite{Feller1971,Grimmett1986}, we obtain
	$$
	\begin{aligned}
		\mathbb{P}(T(\omega)>M_2)
		&=\mathbb{E}[\mathds{1}_{\{T>M_2\}}]
		=\mathbb{E}\left[\mathbb{E}[\mathds{1}_{\{T>M_2\}}\bigg|W(M_2)]\right]\\
		&=\int_{-\infty}^\infty\mathbb{P}\left(T(\omega)>M_2\bigg|W(M_2)=x \right)\frac 1 {(2\pi M_2)^{1/2}}e^{-x^2/(2M_2)}\,{\rm d}x\\
		&=\int_{\beta M_2/4}^\infty \frac 1 {(2\pi M_2)^{1/2}}e^{-x^2/(2M_2)}\,{\rm d}x+\int_{-\infty}^{\beta M_2/4} e^{-\beta^2 M_2/8+\beta x/2}e^{-x^2/(2M_2)}\,{\rm d}x\\
		&=1-\Phi\left(\frac {\beta\sqrt{M_2}} 4\right)+\int_{-\infty}^{\beta M_2/4}\frac 1 {(2\pi M_2)^{1/2}}e^{-(x-\beta M_2/2)^2/(2M_2)}\,{\rm d}x\\
		&=1-\Phi\left(\frac {\beta\sqrt{M_2}} 4\right)+\Phi\left(-\frac {\beta\sqrt{M_2}} 4\right)
		=2\left(1-\Phi\left(\frac {\beta\sqrt{M_2}} 4\right)\right),
	\end{aligned}
	$$
	where $\Phi\left(\frac{a}{\sqrt{b}}\right)$ is the probability that a normally distributed random variable with mean $0$ and variance $b$ is less than or equal to $a$. Similarly, we can deduce that
	\begin{equation}\label{P(T) inequality}
		\mathbb{P}\left(T(\omega)>\left[\frac 1 {2\|u_0\|_{H^1}} \log\frac{R}{C\|u_0\|_{H^{s+\delta}}}-1\right]^{\frac 1 2}\right)=2\left(1-\Phi\left(\frac \beta 4 \left[\frac 1 {2\|u_0\|_{H^1}} \log\frac{R}{C\|u_0\|_{H^{s+\delta}}}-1\right]^{\frac 1 4}\right) \right)
	\end{equation}
	
	Therefore, for fixed $\beta>0$, for any $\varepsilon>0$, combining \eqref{P(A) inequality1}, \eqref{P(A) inequality2} and \eqref{P(T) inequality}, choosing $R$ sufficiently large, we have
	$$\mathbb{P}\left(A(\omega)\exp\left(\|u_0\|_{H^1}(T(\omega)\cdot A(\omega)+1)\right)>\frac R {C\|u_0\|_{H^{s+\delta}}}\right)<\varepsilon,$$
	thus, we can obtain that
	\begin{equation}\label{mu tightness}
		\mu_T(B_R^c)<\frac 1 T\int_0^T\varepsilon\,{\rm d}t=\varepsilon.
	\end{equation}
	Besides, $B_R$ is a compact set of $E_s$ since the embedding from $H^{s+\delta}$ into $H^s$ is compact. Hence, the tightness of the family $\{\mu_T\}_{T>0}$ is proved.
\end{proof}

\begin{proof}[Proof of Theorem \ref{main}]
	The existence of invariant measure for equation \eqref{CH u} is established by Theorem \ref{KB theorem} and Lemma \ref{mu tightness}.
	
	Fixed $\delta>0$, for $u_0,v_0\in E_s$ satisfy $$(1-\partial_x^2)u_0>0,\quad (1-\partial_x^2)v_0<0,\quad u_0,v_0\in H^{s+\delta}$$
	and $u$ and $v$ are the unique global solutions to \eqref{CH u} with initial data $u_{0}$ and $v_0$, respectively. Then Lemma \ref{mu tightness} yields that the time-averaged measures from $u_0$ and $v_0$ converge correspondingly to $\mu^u$ and $\mu^v$.
	
	Besides, it is easily to obtain that for all $t\ge 0$, $(1-\partial_x^2)u>0$ and $(1-\partial_x^2)v<0$. Then for $K=\{u\in E_s:(1-\partial^2)u>0\}$, we have
	$$P_t(u_0,K)=1\quad{\rm and}\quad P_t(v_0,K)=0,\quad\forall\,t\ge0,$$
	i.e. $\mu^u(K)=1$ and $\mu^v(K)=0$. Hence, the invariant measure for equation \eqref{CH u} is not unique. In other words, equation \eqref{CH u} is not ergodic.
\end{proof}

\begin{rema}
	When $\beta=0$, the estimate for $T(\omega)$ in Lemma \ref{mu tightness} no longer holds. Consequently, the invariant measure for the deterministic CH equation cannot be obtained by this method.
\end{rema}

\smallskip
	\noindent\textbf{Acknowledgments}
This work was partially supported by the National Key R\&D Program of China(No. 2021YFA1002100) and the National Natural Science Foundation of China (No.12171493 and No. 12471233).
	
	\smallskip
	\noindent\textbf{Conflict of interest}
	
	The authors have no conflicts to disclose.
	
	\smallskip
	\noindent\textbf{Data Availability}
	
	The data that support the findings of this study are available from the corresponding author upon reasonable request.

	\phantomsection
	\addcontentsline{toc}{section}{\refname}
	\bibliographystyle{abbrv} 
	\bibliography{SCHref}
\end{document}